\documentclass[12pt,reqno]{amsart}
\usepackage{amssymb}
\usepackage{amsmath}
\usepackage{amsthm}
\usepackage{eucal}
\usepackage{color}
\usepackage{amsfonts}
\usepackage[pdftex]{graphicx}
\usepackage[T1]{fontenc}

\newcommand{\R}{\mathbb R}

\newcommand{\sgn}{\text{sgn}}

\newtheorem{theorem}{Theorem}[section]

\newtheorem{remark}[theorem]{Remark}

\newtheorem{corollary}[theorem]{Corollary}

\newtheorem*{TB}{Theorem B}

\newtheorem*{LA}{Lemma A}

\numberwithin{equation}{section}
\begin{document}
\title[Unique continuation]{On the unique continuation of solutions to non-local non-linear dispersive  equations}
\author{C. E. Kenig}
\address[C. E. Kenig]{Department of Mathematics\\University of Chicago\\Chicago, Il. 60637 \\USA.}
\email{cek@math.uchicago.edu}
\author{D. Pilod}
\address[D. Pilod]{Department of Mathematics\\University of Berger\\Postbox 7800\\ 5020 Berger\\Norway.}
\email{Didier.Pilod@uib.no}

\author{G. Ponce}
\address[G. Ponce]{Department  of Mathematics\\
University of California\\
Santa Barbara, CA 93106\\
USA.}
\email{ponce@math.ucsb.edu}

\author{L. Vega}
\address[L. Vega]{UPV/EHU\\Dpto. de Matem\'aticas\\Apto. 644, 48080 Bilbao, Spain, and Basque Center for Applied Mathematics,
E-48009 Bilbao, Spain.}
\email{luis.vega@ehu.es}


\begin{abstract} 
We prove unique continuation properties  of solutions to a large class of nonlinear, non-local dispersive equations. The goal is to show that if $u_1,\,u_2$ are two 
suitable solutions of the equation defined in $\mathbb R^n\times[0,T]$ such that for some non-empty open set $\Omega\subset \mathbb R^n\times[0,T]$, $u_1(x,t)=u_2(x,t)$ for $(x,t)
\in \Omega$, then $u_1(x,t)=u_2(x,t)$ for any $(x,t)\in\mathbb R^n\times[0,T]$. 
The proof is based on  static arguments. More precisely, the main ingredient in the proofs will be  the  unique continuation properties for fractional powers of the Laplacian  established by Ghosh, Salo and Ulhmann in \cite{GhSaUh}, and some extensions obtained here.

\end{abstract}

\maketitle

\section{Introduction}

This work continues the study initiated in \cite{KPV19} concerning  unique continuation properties of solutions to nonlinear non-local dispersive equations.
Roughly, the aim is to prove that if $\,u_1,\,u_2$ are two suitable solutions  of a nonlinear dispersive model for $(x,t)\in\mathbb R^n\times[0,T]$
such that for some open non-empty set $\Omega\subset \mathbb R^n\times [0,T]$ 
\begin{equation}
\label{1}
u_1(x,t)=u_2(x,t),\;\;\;\;\;(x,t)\in \Omega,
\end{equation}
then $\,u_1(x,t)=u_2(x,t)$ for all $(x,t)\in\mathbb R^n\times [0,T]$.

As in \cite{KPV19}  the equations to be examined here have their dispersive relation described by a non-local differential operator. 

Our original motivation is to establish the desired result in real solutions for the so called generalized dispersion  Benjamin-Ono (BO) equation whose initial value problem (IVP) can be written as
\begin{equation}\label{AA}
\begin{cases}
\partial_tu-\partial_xD_x^{\alpha}u +u\partial_xu=0, \hskip5pt \;x, t\in\R, \;\;\;\alpha\in (-1,2)-\{0, 1\},\\
u(x,0)=u_0(x),
\end{cases}
\end{equation}
where in any dimension $n$ the operator $D_x^{\alpha}$ is defined for $\alpha\in(0,2)$ as
\begin{equation}
\begin{aligned}
\label{2}
D_x^{\alpha}(x) &= (-\Delta)^{\alpha/2}f(x) = ((2\pi |\xi|)^{\alpha} \widehat{f})^{\lor}(x)\\
&=\lim_{\epsilon\downarrow 0}\frac{1}{c_{n,\alpha}}\int_{|y|\geq \epsilon}\frac{f(x+y)-f(x)}{|y|^{n+\alpha}}dy,
\end{aligned}
\end{equation}
and for $\alpha\in(-n,0)$ as 
\begin{equation}
\begin{aligned}
\label{2b}
D_x^{\alpha}(x)& = (-\Delta)^{\alpha/2}f(x) = ((2\pi |\xi|)^{\alpha} \widehat{f})^{\lor}(x)\\
&=\frac{1}{c_{n,\alpha}}\int_{\R^n}\frac{f(x+y)}{|y|^{n+\alpha}}dy,
\end{aligned}
\end{equation}
with $$c_{n,\alpha}=\pi^{n/2} 2^{-\alpha}\,\Gamma(-\alpha/2)/\Gamma((n+\alpha)/2),$$ see for example \cite{St}.
 
In \eqref{AA} the limiting cases : $\alpha=2$ corresponds to the Korteweg-de Vries (KdV)  equation \cite{KdV}, $\alpha=1$ corresponds to the BO equation \cite{Be}, \cite{On}, $\alpha=0$, after a change of variable, coincides with the inviscid Burgers' equation which is hyperbolic, and  $\alpha=-1$ is the so called Hilbert-Burgers (HB) equation, see \cite{BiHu} and \cite{HITW}.

The KdV equation and the BO  equation arise  both as mathematical models for the unidirectional propagation of long waves  and in inverse scattering theory. They are completely integrable (see \cite{GGKM} and \cite{AbSe}) the only ones in the family $\alpha\in [-1,\infty)$. In particular, their solutions satisfy infinitely many conservations laws. 

The general case of the equation in \eqref{AA} appears as a model for vorticity waves in the coastal zone, see \cite{ShVo}.

Formally, real valued solutions of \eqref{AA} satisfy three conservation laws:
\begin{equation}
\label {CL}
\begin{aligned}
I_1(u)&=\int_{-\infty}^{\infty}u(x,t)dx,\;\;\;I_2(u)=\int_{-\infty}^{\infty}u^2(x,t)dx,\\
I_3(u)&=\int_{-\infty}^{\infty} ((D_x^{\alpha/2}u)^2-\frac{1}{3}u^3)(x,t)dx.
\end{aligned}
\end{equation}

The well-posedness of the IVP \eqref{AA} has been extensively considered. For the case $\alpha\in(1,2)$ we refer to  \cite{GiVe}, \cite{FLP}, \cite{CoKeSt}, \cite{Gu}, \cite{MoSaTz}, \cite{MoRi}, \cite{MoVe},  \cite{HIKK} and references therein, and for  $\alpha\in(0,1)$ we refer to \cite{LiPiSa}, \cite{MoPiVe} and references therein, (for the case $\alpha=1$ see the reference  \cite{KPV19}).

In \cite{HIKK} it was shown that the IVP \eqref{AA} with $\alpha\in (1,2)$ is globally well-posed in $L^2(\mathbb R)$.
In \cite{MoPiVe} it was established that the IVP \eqref{AA} with $\alpha\in(0,1)$ is locally well-posed in $H^s(\mathbb R),\;s>3/2-5\alpha/4$ , and globally well-posed in $H^{\alpha/2}(\mathbb R)$ if $\alpha>6/7$. 

In the case $\alpha\in (-1,0)$ the local well posedness in $H^s(\R)$ with $ s>3/2$ follows by a standard argument based on energy estimates.

Our first result  is the following :

\begin{theorem}\label{A1}
Let $\,\alpha\in (-1,2)-\{0, 1\}$. Let $\,u_1,\,u_2$ be two real solutions of the IVP \eqref{AA} such that
\begin{equation}
\label{class1}
u_1,\,u_2\in C([0,T]:H^{s}(\mathbb R))\cap C^1([0,T]:H^{s'}(\mathbb R)),
\end{equation}
with $s>\max\{\alpha+1; 3/2\}$ and  $s'=min\{s-(\alpha+1);s-1\}$. Moreover, if $\alpha\in (-1,-1/2]$ assume that
\begin{equation}\label{excep}
x\partial_xu_1,\;x\partial_xu_2\in L^{\infty}([0,T]:L^2(\mathbb R)).
\end{equation} 
If there exits a non-empty open set $\Omega\subset R\times[0,T]$ such that
\begin{equation}
\label{con1}
u_1(x,t)=u_2(x,t),\;\;\;\;\;(x,t)\in \Omega,
\end{equation}
then $\,u_1(x,t)=u_2(x,t)$ for all $(x,t)\in\mathbb R\times [0,T]$

\end{theorem}

\begin{remark} 
\label{111}

(i) For the value $\alpha=2$, i.e. the KdV equation, the result in Theorem \ref{A1} was proved by Saut-Scheurer in \cite{SaSc}. In this case,  the operator modeling the dispersion relation is local,  and the proof is based on appropriate Carleman estimates. 

In the cases $\alpha=\pm 1$, i.e. for the BO and HB  equations resp., the result in Theorem \ref{A1} was recently obtained in \cite{KPV19}. 

In the case $\alpha=0$ the result fails. 

\vskip.1in
(ii) The hypothesis in \eqref{class1} guarantees that $\partial_xu_j(\cdot,t),\,j=1,2$, $\,t\in[0,T]$ is a continuous function and that 
$ \partial_tu(\cdot,t)\in L^2(\mathbb R),\,j=1,2$ for any $\,t\in[0,T]$. This is satisfied if $u_1(x,0), \,u_2(x,0)\in H^s(\mathbb R)$, i.e. the solution flow preserves (locally in time) the class $H^s(\mathbb R)$ with $s>max\{\alpha+1; 3/2\}$.

The hypothesis \eqref{excep} implies that for any $t\in [0,T]$ one has that $\partial_x u_1(t),\,\partial_xu_2(t)\in L^1(\mathbb R)$ which is essential in the proof of Theorem \ref{A1} for the case $\alpha\in (-1,-1/2]$. Property \eqref{excep} is ensured if one assumes that $$x \partial_xu_1(x,0), \,x \partial_xu_2(x,0)\in L^2(\mathbb R).$$ This will be proved in the appendix (section 4).

\vskip.1in
(iii) It will be clear from our proof below that the result in Theorem \ref{A1} extends to the equation in \eqref{AA} with a more general non-linearity. In fact, Theorem \ref{A1}  applies to any pair of appropriate  solutions $u_1,\,u_2$ of the IVP associated to  the equation
\begin{equation}
\label{gen1}
\partial_tu-\partial_xD_x^{\alpha}u +F(u,..., \partial_x^ju)=0, \hskip5pt \;x, t\in\R,
\end{equation}
with $\alpha\in [-1,\infty)-2\mathbb Z$, $j\in\mathbb N$ and $F(\cdot)$ a regular enough function representing the non-linearity.

\vskip.1in

(iv) For some non-local models, our version of unique continuation is too strong, and it is only realized by assuming that $u_2\equiv 0$.   This is the case of the IVP associated to the Camassa-Holm equation \cite{CaHo}
\begin{equation}\label{CH}
\partial_tu+ u \partial_xu+\partial_x(1-\partial_x^2)^{-1}\big(u^2+\frac{1}{2} (\partial_xu)^2\big)=0, \hskip5pt \;t,\,x\in\R,
\end{equation}
see \cite{LiPo}. This will also occur in part (ii) of Theorems \ref{A2}-\ref{A2m} below for the equation \eqref{fNLS} when a Hartree term is present.

\end{remark}

An argument similar to that used in the proof of Theorem \ref{A1} will also yield the following result :

\begin{corollary}\label{A17}
Let $\,\alpha\in (-1,2)-\{0\}$. Let $\,u $ be a real solution of the IVP \eqref{AA} in the class defined in \eqref{class1}. Moreover, if $\alpha\in (-1,-1/2]$ assume that
\begin{equation}\label{excep1}
x\partial_xu\in L^{\infty}([0,T]:L^2(\mathbb R)).
\end{equation} 
If there exit a constant $c_0\in \mathbb R$ and  a non-empty open set $\Omega\subset R\times[0,T]$ such that
\begin{equation}
\label{con1&}
u(x,t)=c_0,\;\;\;\;\;(x,t)\in \Omega,
\end{equation}
then $c_0=0$ and $\,u(x,t)=0$ for all $(x,t)\in\mathbb R\times [0,T]$.

\end{corollary}

\begin{remark}  (i) By working with $\partial_xu$ instead of $u$ the result in Corollary \ref{A17} extends to the end-point cases $\alpha=2$ (KdV) and $\alpha=-1$ (HB),
for  appropriate solutions, using the arguments in \cite{SaSc} and \cite{KPV19}, respectively.

(ii) In the case $\alpha=2$ (KdV) under more restricted hypothesis stronger results are known. More precisely, in \cite{Ta} it was shown that the solution of the IVP for the KdV equation corresponding to data 
$u_0\in L^2(\mathbb R)$ with the additional decay assumption
$$
\int_{-\infty}^{\infty} (1+|x|)|u_0(x)|dx +\int_{0}^{\infty} e^{\delta |x|^{1/2}} \,|u_0(x)|dx<\infty \;\;\;\text{for some} \;\;\delta>0,
$$
 becomes analytic with respect to the variable $x$ for all $t>0$. This result is based on the inverse scattering method and it is unknown for other nonlinearities. 

\end{remark}
\vskip.1in

 In the case of $\alpha=\pm 1$, i.e. for BO and HB  equations, respectively, the proof of Theorem \ref{A1} in \cite{KPV19} is based on the identities 
$$
(-\partial_x^2)^{1/2}=D_x= \partial_x \mathcal H\;\;\;\;\;\;\;\text{and} \;\;\;\;\;\;\;\partial_x D_x^{-1}=\mathcal H,
$$
 where $\mathcal H$ is the Hilbert transform
$$
\mathcal Hf(x)=\frac{1}{\pi}\,\rm{p.v.}\Big(\frac{1}{x}\ast f\Big)(x)=(-i\,\sgn(\xi)\widehat{f}(\xi))^{\vee}(x),
$$
and  the following well-known unique continuation principle for the Hilbert transform :

\begin{LA}
\label{L0}
Let $f\in L^2(\R)$ be a real valued function. If there exists an open (non-empty) interval  $I\subset \R$ such that $$f(x)=\mathcal Hf(x)=0,\;\;\;\;\;\text{for a. e.}\;\;\;x\in I,
$$
then $f(x)=0$ for any  $x\in\R$.

\end{LA}

A key argument in the proof of Theorem \ref{A1} will be  following unique continuation result regarding the fractional powers of the Laplacian in $\mathbb R^n$ established by Ghosh-Salo-Uhlmann in \cite{GhSaUh}:

\begin{TB} [\cite{GhSaUh}]
 Let $\alpha\in (0,2)$ and $\,f\in H^s(\mathbb R^n)$ for some $\,s\in\mathbb R$. If there exists a open (non-empty) set $\Theta\subset \mathbb R^n$ such that 
\begin{equation}
\label{hyp}
(-\Delta)^{\alpha/2}f (x)=f(x)=0,\;\;\;\;\text{ in}\,\;\;\mathcal D'(\Theta),
\end{equation}
 then $f\equiv 0$ in $H^s(\mathbb R^n)$.

\end{TB}

\begin{remark}
\label{rem1}
 It is clear that the statement of Theorem B extends to  $\alpha\in (0,\infty)-2 \mathbb Z$, and by \lq\lq duality" to 
$\alpha \in (-n/2,\infty)-2 \mathbb Z$, (see the argument at the end of the proof of Theorem \ref{A1} in section 2). 

Moreover, if $f\in L^1(\mathbb R^n)\cap L^2(\mathbb R^n)$, then it also holds for $\alpha\in (-n,-n/2]$, see \cite{CoMoRa} for less restrictive regularity assumptions.

\end{remark}

Unique continuation for the fractional Laplacian has been studied in other works, see \cite{Ri}, \cite{Ru}, \cite{FaFe}, \cite{Yu}, \cite{CoMoRa} and references therein.

To illustrate the argument to prove Theorem B and its extensions it is convenient to examine the case $\alpha=1$ and $n\geq 1$  (when $n=1$ the result follows by Lemma A) assuming that $f\in H^{1/2}(\mathbb R^n)$.
By inspecting the elliptic extension problem
\begin{equation}
\label{simple}
\begin{aligned}
\begin{cases}
&(\Delta_x+\partial_y^2)U(x,y)=0,\;\;\;\;(x,y)\in\mathbb R^n\times(0,\infty),\\
&U(x,0)=f(x),
\end{cases}
\end{aligned}
\end{equation}
one has that
\begin{equation}
\label{def1}
(-\Delta_x)^{1/2}f(x)=-\,\partial_yU(x,0).
\end{equation}
Then, from the hypothesis \eqref{hyp} using the reflexion principle for harmonic functions one has that  $U(x,y)$ has a harmonic extension $\widetilde U(x,y)$ defined in the open set $ \mathbb R^n\times (-\infty,0)\cup \Theta\times\{0\}\cup\mathbb R^n\times(0,\infty).$ From the equation, the identity \eqref{def1}  and the hypothesis \eqref{hyp} it follows that $\tilde U(x,y)$  has zero of infinite order at any point of $\Theta\times\{0\}$. Hence, since $\widetilde U$ is real analytic, it follows that $\widetilde U(x,y)\equiv 0$ and consequently $f(x)= 0$ for $x\in\mathbb R^n$, which proves Theorem B for $\alpha=1$ under the assumption $f\in H^{1/2}(\mathbb R^n)$.

In the general case $\alpha\in(0,2)$ the proof of Theorem B, with $f\in H^{\alpha}(\mathbb R^n)$ follows by combining the ideas of Cafarelli-Silvestre  in \cite{CaSi}, with a unique continuation principle  obtained in  \cite{Ru}, 
see also \cite{FaFe} and \cite {Yu}, and some regularity results found in \cite{CaSi}, see \cite{GhSaUh}.

Further, we notice that the argument given above to prove Theorem  B in the particular  case $\alpha=1/2$ still applies if one considers $(1-\Delta)^{1/2}$ instead of $(-\Delta)^{1/2}$. More precisely, one regards the elliptic problem  
\begin{equation}
\label{simple11}
\begin{aligned}
\begin{cases}
&(\Delta_x-1+\partial_y^2)U(x,y)=0,\;\;\;\;(x,y)\in\mathbb R^n\times(0,\infty),\\
&U(x,0)=f(x),
\end{cases}
\end{aligned}
\end{equation}
to get that
\begin{equation}
\label{def11}
(1-\Delta_x)^{1/2}f(x)=-\,\partial_yU(x,0).
\end{equation}
Assuming that $f\in H^{1/2}(\mathbb R^n)$ and $\,(1-\Delta_x)^{1/2}f (x)=f(x)=0$ for $x\in \Theta$ open (non-empty) sub-set of $\R^n$  the above approach  yields the desired result $\,f(x)= 0$ for any $x\in\mathbb R^n$. 

For the general case, we shall prove the following  related version of Theorem B  :

\begin{theorem}\label{C1}
Let $\alpha\in (0,2) $ and $\,f\in H^{s}(\mathbb R^n)$ for some $s\in\mathbb R$. If there exists a open (non-empty) set $\Theta\subset \mathbb R^n$ such that 
\begin{equation}
\label{hyp2}
(1-\Delta)^{\alpha/2}f (x)=f(x)=0,\;\;\;\;\text{in}\;\;\mathcal D'(\Theta),
\end{equation}
 then $f(x)\equiv  0$.

\end{theorem}

\begin{remark} 
(i) The range of  $\alpha$ in Theorem \ref{C1} can be extended to $\mathbb R-2\mathbb Z$. 

(ii) Our proof of Theorem \ref{C1} starts by using the  abstract approach in \cite{StTo} to formulate  a general setting and then uses some ideas which  reduce the proof of this theorem to that 
of Theorem B in a higher dimension. This argument is quite natural and enables us to consider the extension of this problem to fractional powers of elliptic operators with variable coefficients in $\mathbb R^n$ and $\mathbb S^n$ as well as in manifold. This will be addressed   in a forthcoming work.
\end{remark}

 Theorem B and Theorem \ref{C1} allow us to extend the unique continuation result in Theorems \ref{A1} to a large class of non-linear non-local higher dimensional dispersive models. More precisely, we consider  the general fractional Schr\"odinger equation 
 \begin{equation}
 \label {fNLS}
 i\partial_t u +(\mathcal L_m)^{\alpha/2}u+V u+(W\ast F(|u|)u+P(u,\bar{u})=0,
 \end{equation}
where $(x,t)\in \mathbb R^n\times\mathbb R$,  $\mathcal L_m=(-\Delta+m^2),\;m\geq 0$, with $\alpha\in\R-2\,\mathbb Z$ if $m>0$ and $\alpha\in (-n,\infty)-\mathbb Z$ if $m=0$, $V=V(x,t)$ is the potential energy,  $W=W(|x|)$ defines the Hartree integrand and $P(z,\bar{z})$ is a polynomial or a regular enough function with $P(0,0)=0$.

The model in \eqref{fNLS} arises in several different physical contexts, for example :

- when $m=0$, $W=P=0$, it was used in \cite{La} to describe particles in a class of  Levi stochastic processes,

- when $m>0$, $W=P=0$, it was derived as a generalized semi-relativistic (Schr\"odinger) equation, see \cite{Le} and references therein,

-  when $m=0$, $\alpha=1$, $V=W=0$, and $P(u,\bar{u})=\pm|u|^au,\,a>0,$ it is known as the half-wave equation, see \cite{BGLV}, \cite{GeLePoRa} and reference therein,

-  when $m=0$,  $V=W=0$, and $$P(z,\bar z)=c_0|z|^2z+c_1z^3+c_3z{\bar z}^2+c_3{\bar z}^3,\;\;\,c_0\in\mathbb R,\;c_1,c_2, c_3\in \mathbb C, $$it appears in \cite{IoPu} on the study of the long-time behavior of solutions to the water waves equations in $\mathbb R^2$, where $(-\partial_x^2)^{1/2}$ 
 models the dispersion relation of the linearized gravity water waves equations for one-dimensional interfaces,
 
 - when $m=1$, $V=P=0$, $F(|z|)=|z|^2$ the model arises in gravitational collapse, see \cite{ElSc}, \cite{FrLe} and references therein.

 The well-posedness of the IVP \eqref{fNLS} has been considered in several publications, see for example \cite{CHHO}, \cite{HoSi}, \cite{KLR}, \cite{Le}, \cite{BoRi} and references therein.
  
 Thus, we have the following results concerning unique continuation properties of solutions to the IVP associated to the equation \eqref{fNLS}. First, we shall consider the case $m>0$ :
 
\begin{theorem}\label{A2}
(i) Let $\,\alpha\in (-2,2)-\{0\}$ and $m>0$. Let $\,u_1,\,u_2$ be two solutions of the IVP associated to the equation  \eqref{fNLS} with $W\equiv 0$ such that
\begin{equation}
\label{class2}
u_1,\,u_2\in C([0,T]:H^{s}(\mathbb R^n))\cap C^1([0,T]:H^{s-\alpha}(\mathbb R^n)),
\end{equation}
with $s>\max\{\alpha; n/2\}.$ If there exits an open  non-empty set $\Omega\subset \mathbb R^n\times[0,T]$ such that
\begin{equation}
\label{con2}
u_1(x,t)=u_2(x,t),\;\;\;\;\;(x,t)\in \Omega,
\end{equation}
then $\,u_1(x,t)=u_2(x,t)$ for all $(x,t)\in\mathbb R^n\times [0,T]$.

\vskip.1in

(ii) For the general form of the equation \eqref{fNLS} the result in (i) still holds if one assumes that $u_2(x,t)=0$ for $(x,t)\in \mathbb R^n\times [0,T]$.

\end{theorem}

Next, we consider the case $m=0$:

\begin{theorem}\label{A2m}
(i) Let $\,\alpha\in (0,2)$ and $m=0$. Let $\,u_1,\,u_2$ be two solutions of the IVP associated to the equation  \eqref{fNLS} with $W\equiv 0$ such that
\begin{equation}
\label{class2m}
u_1,\,u_2\in C([0,T]:H^{s}(\mathbb R^n))\cap C^1([0,T]:H^{s-\alpha}(\mathbb R^n)),
\end{equation}
with $s>\max\{\alpha; n/2\}$. If there exits a non-empty open set $\Omega\subset \mathbb R^n\times[0,T]$ such that
\begin{equation}
\label{con2m}
u_1(x,t)=u_2(x,t),\;\;\;\;\;(x,t)\in \Omega,
\end{equation}
then $\,u_1(x,t)=u_2(x,t)$ for all $(x,t)\in\mathbb R^n\times [0,T]$.

\vskip.1in

(ii) For the general form of  the equation \eqref{fNLS} the result in (i) still holds if one assumes that $u_2(x,t)=0$ for $(x,t)\in \mathbb R^n\times [0,T]$.
\end{theorem}

\begin{remark}
\label{112}

(i) In Theorems \ref{A2} -\ref{A2m} we are assuming the existence of solutions of the IVP \eqref{fNLS}  in the class described in \eqref{class2} and \eqref{class2m}. This depends on  the dimension $n$, on  the values of $m\geq 0$ and $\alpha\in(0,2)$, on the regularity and decay properties of $V$ and $W$, and on the structure of the terms $F(\cdot)$ and $P(\cdot,\cdot)$. In particular, the hypothesis  $s>n/2$ guarantees that $H^s(\mathbb R^n)$ 
is an algebra respect to the usual product of functions so the polynomial $P(u,\overline u)$ belongs to $L^2(\mathbb R^n)$ at each $t\in[0,T]$
\vskip.1in

(ii) The results in Theorem \ref{A2m} ($m =0$) extend to the case $ \alpha\in (-n/2,0)-\mathbb Z$ by assuming (instead of \eqref{class2m}) that
$$
u_1, u_2\in C([0,T]:H^{}(\mathbb R^n))\cap C^1([0,T]:L^p(\mathbb R^n)),
$$
with $s>n/2$ and $1/p=1/2+\alpha/2.$

\vskip.1in 
(iii) As  was mentioned before, in part (ii) of Theorems \ref{A2} and \ref{A2m} we  achieve a weaker version of the desired unique continuation result. Roughly speaking, in the general case 
one works with the equation solved by $w(x,t)=(u_1-u_2)(x,t)$. However,  when the non-local Hartree term in \eqref{fNLS} is present, the equation for $w(x,t)$ does not satisfy the necessary property  enjoyed by  a solution  $v$ of \eqref{fNLS}, that is, for any $\Omega\subset \mathbb R^n\times [0,T])$ open set
$$
\text{if}\,\;v(x,t)=0,\,\,(x,t)\in\Omega,\,\,\text{then}\;\,(\mathcal L_m)^{\alpha/2}v(x,t)=0,\,\,(x,t)\in\Omega.
$$
\vskip.1in
(iv) The comments in Remark \ref{111} part (iii) concerning more general nonlinearities  in \eqref{fNLS} also apply to Theorems \ref{A2} and \ref{A2m}.
\vskip.1in

\end{remark}

Next, we consider the extension of Corollary \ref{A17} to solutions of the IVP associated to the equation in \eqref{fNLS}.

\begin{corollary}\label{A18}
 Let $\,\alpha\in (-2,2)-\{0\}$ if  $m>0$ and  $\,\alpha\in (0,2)-\{0\}$ if $m=0$. Let $\,u $ be a solution of the IVP associated  to the equation \eqref{fNLS} with  $V\equiv W\equiv 0$ such that 
\begin{equation}\label{excep2}
u\in C([0,T]:H^{s}(\mathbb R^n))\cap C^1([0,T]:H^{s-\alpha}(\mathbb R^n)).
\end{equation} 
If there exit a constant $c_0\in \mathbb R$ and  a non-empty open set $\Omega\subset R\times[0,T]$ such that
\begin{equation}
\label{con2&}
u(x,t)=c_0,\;\;\;\;\;(x,t)\in \Omega,
\end{equation}
then $c_0=0$ and $\,u_1(x,t)=0$ for all $(x,t)\in\mathbb R\times [0,T]$.  


\end{corollary}

\begin{remark} The observations in Remark \ref{112} (ii) also  hold for  Corollary \ref{A18} in the case $m=0$.

\end{remark}

Next,  we shall give an application of Theorem B to a non-local non-dispersive model. We consider the $2$-D quasi-geostrophic (QG) equation
\begin{equation}
\label{QG} 
\partial_t\theta+u\cdot \nabla\theta=0,\;\;\;\;\;\;\;(x,t)\in\mathbb R^2\times\mathbb R,
\end{equation}
 and
\begin{equation}
\label{qg1}
u=(u_1,u_2)=\nabla^{\perp}(-\Delta)^{-1/2}\theta=(-\partial_y,\partial_x)(-\Delta)^{-1/2}\theta,
\end{equation}
where $\theta$ denotes a scalar (potential temperature) convected by the velocity field $u$.

The QG was derived in \cite{Pe} as a model for atmospheric turbulence. Its relation with the Euler system for incompressible (homogeneous) flow 
$$
\begin{aligned}
\begin{cases}
&\partial_tv+v\cdot \nabla v=\nabla p,\;\;\;k=1,..,n,\;\;n=2,3,\\
&\nabla\cdot v=0,
\end{cases}
\end{aligned}
$$
where $v=(v_1,..,v_n)$ is the velocity and $p$ denotes the pressure, and its vorticity $\omega =\nabla\times v$ formulation in $3$-D
\begin{equation}
\label{Vor3}
\partial_t\omega+v \cdot \nabla \omega=(Dv)\omega,
\end{equation}
has been extensively studied, see \cite{CoMaTa} and \cite{CoGSIo} and references therein. Notice that by differentiating the equation in \eqref{QG} one gets
\begin{equation}
\label{abc}
\partial_t(\nabla^{\perp}\theta)+u \cdot \nabla (\nabla^{\perp}\theta)=(Du)(\nabla^{\perp}\theta),
\end{equation}
which resembles to the equation in \eqref{Vor3}.

We also observe that in the $2$-D case of the Euler system the vorticity $\omega$ is a scalar function satisfying the equation
\begin{equation}
\label{Vor2}
\partial_t \omega+v\cdot \nabla \omega=0,
\end{equation}
which resembles that  in \eqref{QG}. In both dimensions $2$ and $3$ the Biot-Savart law 
\begin{equation}
\label{BS}
v=\nabla\times (-\Delta)^{-1}\omega,
\end{equation}
allows to recover the velocity $v$ in terms of the vorticity $\omega$.

The following theorem presents a difference between the behavior of the solutions of the QG equation and those for the vorticity system in $2$ and $3$ dimensions.

\begin{theorem}\label{X}
Let $\theta$ be a  solution of the IVP associated to the equation  \eqref{QG} such that
\begin{equation}
\label{classQG}
\theta\in C([0,T]:H^{s}(\mathbb R^2))\cap C^1([0,T]:H^{s-1}(\mathbb R^2)),
\end{equation}
with $s>1$. If there exit $t_0\in[0,T]$,  $c_0\in\mathbb R$ and  a non-empty  open set $\Theta\subset \mathbb R^2$ such that
\begin{equation}
\label{conQG}
\theta(\cdot,t_0)\big |_{\Theta}=c_0\;\;\;\text{and}\;\;\;u(\cdot,t_0)\big |_{\Theta}=(0,0),
\end{equation}
then $c_0=0$, $\theta(x,t)=0$ and $u(x,t)=(0,0)$  for all $(x,t)\in\mathbb R^2\times [0,T]$.

\end{theorem}

\begin{remark} (i) We observe  that Theorem \ref{X} fails for the Euler system in $2$ and $3$ dimensions by just picking up  a data $v(\vec x, 0)=\vec c_0$ for  $\vec x\in \Theta$, (since $\omega(\vec x,0)=\nabla\times v(\vec x,0)$), and using that the equations are time reversible. As it was shown above in the case $n=3$ one can think of $\nabla^{\perp}\theta$ as the vorticity and in the case $n=2$ one sees $\theta$ as the vorticity itself.

(ii) The proof of the  existence of solution (in fact,  the local well-posedness in $=H^s(\mathbb R^n)$, with $s>n/2+1$) to the IVP associated to the equation\eqref{QG} in the class \eqref{classQG} follows a classical argument. 

(iii) Theorem \ref{X} applies to the generalized form of the QG equation considered in \cite{CoGSIo}
\begin{equation}
\label{QG1} 
\partial_t\theta+u\cdot \nabla\theta=0,\;\;\;\;\;\;\;\;u=(u_1,u_2)=\nabla^{\perp}(-\Delta)^{-1+\alpha/2}\theta,
\end{equation}
with $\alpha\in(0,2)$, where $\alpha=1$ corresponds to the QG equation and $\alpha=0$ to the vorticity equation in $2$-D \eqref{Vor2}.

\end{remark}

\vskip.1in
The rest of this paper is organized as follows : section 2 contains  the proof of Theorem \ref{C1}. The proof of Theorem \ref{A1}, Corollary \ref{A17}, Theorems  \ref{A2}-\ref{A2m}, Corollary \ref{A18} and Theorem \ref{X} will be in section 3. Section 4 (appendix) contains the proof of the last statement in Remark \ref{111} part (ii) concerning condition \eqref{excep}.

\section{Preliminary estimates}
\begin{proof} [Proof of Theorem \ref{C1}]

First note that, without loss of generality, we can assume that $f\in H^{\infty}(\mathbb R^n)$, by convolving with a compactly supported approximate identity. 

Let $\alpha\in (0,2)$ and $\Theta$ be a non-empty open sub-set of $\,\mathbb R^n$. Let $L$ be a non-negative second order elliptic operator with a dense domain $dom(L)$ in $L^2(\mathbb R^n)$. Following \cite{StTo} for $f\in dom(L)$ we define for $y\in[0,\infty)$
\begin{equation}
\label{21}
\begin{aligned}
V(x,y)&=\frac{1}{\Gamma(\alpha/2)}\,\int_0^{\infty} e^{-tL}\,(L^{\alpha/2}f)(x)\,e^{-y^2/4t}\,\frac{dt}{t^{1-\alpha/2}}\\
&=\frac{y^{\alpha}}{2^{\alpha} \Gamma(\alpha/2)}\,\int_0^{\infty} e^{-tL}f(x)\,e^{-y^2/4t}\,\frac{dt}{t^{1+\alpha/2}}\\
\end{aligned}
\end{equation}

Thus, by Theorem 1.1 in  \cite{StTo} one has 
\begin{equation}
\label{22}
V \in C^{\infty}((0,\infty): dom(L)) \cap C^{\infty}([0,\infty):L^2(\mathbb R^n)),
\end{equation}
with
 \begin{equation}
 \label{ellip1}
 \begin{aligned}
 \begin{cases}
 & - L_xV+\dfrac{1-\alpha}{y}\,\partial_yV+\partial_y^2V=0,\;\;\;\;\;\;(x,y)\in\mathbb R^n\times(0,\infty),\\
 & V(x,0)=f(x).
 \end{cases}
 \end{aligned}
 \end{equation}
 and 
\begin{equation}
\label{def7}
L^{\alpha/2}f(x)=  - c^*_{\alpha}\,\lim_{y\downarrow 0} \,y^{1-\alpha} \,\partial_yV(x,y),
\end{equation}
where $\,c^*_{\alpha}=2^{\alpha}\Gamma(\alpha/2)/\alpha\Gamma(-\alpha/2)$.

 As it was remarked in \cite{StTo} the identities above should be understood in $L^2(\mathbb R^n)$. Moreover,  $V(x,y)$  solves the singular boundary value problem \eqref{ellip1} and the first equality in \eqref{21} involves the heat semi-group $e^{-tL}$ acting on $L^{\alpha/2}f$ and the second does not involve fractional powers of $L$.  Thus,
 \eqref{def7} describes the fractional powers of $L$ in terms of the solution of the boundary value  \eqref{ellip1} which is the original idea in \cite{CaSi}.

Notice that in the case $L=-\Delta_x$ the equation in \eqref{ellip1} can be written in the divergence form
\begin{equation}
\label{ellip2}
div_{x,y}(y^{1-\alpha}\nabla_{x,y}V)=0.
\end{equation}
and that when $\alpha=1$ \eqref{ellip1} and \eqref{def7} agree with \eqref{simple} and \eqref{def1}, respectively.

 Next, we fix the operator $L=1-\Delta_x$, assume without loss of generality that $\Theta=B_1(0)=\{x\in\mathbb R^n\;:\,|x|<1\}$
 and $f\in H^{\infty}(\mathbb R^n)$ with
 \begin{equation}
 \label{hyp22}
 f\Big |_{B_1(0)}=(1-\Delta)^{\alpha}f \Big |_{B_1(0)}\equiv 0.
 \end{equation}
 In this case 
 \begin{equation}
 \label{ellip3}
 \begin{aligned}
 \begin{cases}
 & \Delta_xV-V+\dfrac{1-\alpha}{y}\,\partial_yV+\partial_y^2V=0,\;\;\;(x,y)\in \mathbb R^n\times(0,\infty),\\
 & V(x,0)=f(x),
 \end{cases}
 \end{aligned}
 \end{equation}
 with
 \begin{equation}
 \label{222}
 (1-\Delta_x)^{\alpha/2}f(x)= - c^*_{\alpha}\,\lim_{y\downarrow 0} \,y^{1-\alpha} \,\partial_yV(x,y),\;\;\;\;x\in\mathbb R^n.
\end{equation}

 We note that under our regularity assumptions and the hypothesis \eqref{hyp22}  the local infinite order of vanishing argument in \cite{Ru} (proof of Proposition 2.2) applies  to the equation \eqref{ellip3}. In particular, it tells us that the solution  $V(x,y)$ vanishes of infinite order at $(x,y)\in B_1(0)\times \{0\}$, i.e. 
 \begin{equation}
 \label{vanishing}
 \lim_{y\downarrow 0}\;\frac{V(x,y)}{y^k}=0,\;\;\;\;\;\forall \,x\in B_1(0)\;\;\forall k\in\mathbb N.
 \end{equation}

Now we define the function $\tilde V(x,t,y)$ as
\begin{equation}\
\label{key}
\tilde V(x,t,y)=\cos(t)\,V(x,y),\;\;\;(x,t,y)\in\mathbb R^{n+1}\times [0,\infty).
\end{equation}

Therefore,
\begin{equation}
 \label{ellip4}
 \begin{aligned}
 \begin{cases}
 & \Delta_{x,t,y}\tilde V+\dfrac{1-\alpha}{y}\,\partial_y\tilde V=0,\;(x,t,y)\in \mathbb R^{n+1}\times(0,\infty),\\
 & \tilde V(x,t,0)=\cos(t)\,f(x),
 \end{cases}
 \end{aligned}
 \end{equation}
with
\begin{equation}
 \label{vanishing2}
 \lim_{y\downarrow 0}\;\frac{\tilde V(x,t, y)}{y^k}=0,\;\;\;\;\;\forall \,(x,t)\in B_1(0)\times(-1,1)\;\;\forall k\in\mathbb N.
 \end{equation}

Thus, we are in the setting of the proof of Theorem B. Applying Proposition 2.2. in \cite{Ru} one gets that 
$$\tilde V(x,t,y)=0\;\;\;\text{for}\;\;(x,t,y)\in\{(x,t,y)\,:\;|x|^2 +t^2+y^2<1,\;y>0\}.
$$
But $\tilde V$ is real analytic in $\mathbb R^{n+1}\times (\epsilon,\infty)$ for any $\epsilon>0$. Therefore $\tilde V\equiv 0$  and so $V\equiv 0$ and consequently $f(x)=0$ for any $x\in\mathbb R$.

 This completes the proof for the case $\alpha\in (0,2)$.

Next, for $\alpha\in (0,\infty)-2\mathbb Z$ and $f\in H^{s}(\mathbb R^n)$, $s\in\mathbb R$, we first reduce it to the case $f\in H^{\infty}(\mathbb R^n)$. Thus, applying the result for  $\alpha\in (0,2)$  to
$$
\tilde f(x)=(1-\Delta)^{k}f(x),\;\;\;\;\;\;\text{with}\;\;\;\;\;k\in \mathbb N,\;\;\text{and}\;\;\;\alpha\in (2k,2(k+1)),
$$
the desired result follows.

Finally, if $\alpha\in (-\infty,0)-2\mathbb Z$ we define
$$
g(x)=(1-\Delta)^{\alpha/2}f(x).
$$

Thus, from hypothesis \eqref{hyp2} one has that
$$
g(x)=0 \;\;\;\;\text{and}\;\;\;\;(1-\Delta)^{-\alpha/2} g(x)=f(x)=0,\;\;\;\;x\in\Theta,
$$
with $-\alpha\in (0,\infty)-2\mathbb Z$ and $g\in H^{-\alpha}(\mathbb R^n)$. Therefore, from the previous case one has that $g(x)=0$ for $x\in\mathbb R^n$ and consequently, $f(x)=0$ for $x\in\mathbb R^n$.

\end{proof}
\section{Proof of Theorems \ref{A1}-\ref{A2}-\ref{A2m}-\ref{X} and Corollaries \ref{A17}-\ref{A18}}
\begin{proof} [Proof of Theorem \ref{A1}]

 We define the function 
$$
w(x,t)=(u_1-u_2)(x,t)
$$
which satisfies the equation
\begin{equation}\label{AAA}
\partial_tw-\partial_xD_x^{\alpha}w +\partial_xu_1 w+u_2\partial_xw=0, \, \;x, t\in\R, \;\alpha\in (-1,2)-\{0\}.
\end{equation}
From the hypothesis \eqref{con1} and the equation \eqref{AAA} it follows that
\begin{equation}
\label{AB1}
w(x,t)=D_x^{\alpha} \partial_x w(x,t) =0,\;\;(x,t)\in\Omega\subseteq \mathbb R\times [0,T].
\end{equation}
Hence, there exist  $\,t_0\in(0,T)$ and an open non-empty interval $(a,b)$ such that $\,(a,b)\times\{t_0\}\subset \Omega$.

First, we observe by he Hardy-Littlewood-Sobolev inequality that if $\alpha\in (-1/2,2)$ and $f(x)=\partial_x w(x,t_0)$, then $D_x^{\alpha}f$ is a well defined function in $L^p(\mathbb R)$ with $p=2$ if $\alpha\in[0,2)$ and with $p$ such that $1/p=1/2+\alpha$ if $\alpha\in(-1/2,0)$.

Therefore, applying Theorem B and Remark \ref{rem1} to the function $\,\partial_xw(x,t_0)\,$ one obtains that $\,\partial_xw(x,t_0)=0$ for any $\,x\in\mathbb R$ and consequently $w(x,t_0)=0$ for any $x\in \mathbb R$  which yields the desired result in the case $\alpha\in (-1/2,2)$.

Next, we consider the case $\alpha\in(-1,-1/2]$. By the assumptions \eqref{class1} and  \eqref{excep} it follows that if $f(x)=\partial_xw(x,t_0)$, then $f\in L^1(\mathbb R)\cap L^2(\mathbb R)$. Thus, $D_x^{\alpha}f\in L^q(\mathbb R)$ for any $q\in (1/(1+\alpha), 2/(1+2\alpha))$. 

Hence, applying Theorem B and Remark \ref{rem1} to the function $\,\partial_xw(x,t_0)\,$ one gets that $\,w(x,t_0)=0$ for any $\,x\in\mathbb R$, which completes the proof.

\end{proof}

\begin{proof} [Proof of Corollary \ref{A17}]

By applying the argument in the proof of Theorem \ref{A1} to the equation \eqref{AA} it follows that $f(x)=\partial_xu(x,t_0)=0$ for all $x\in\mathbb R^n$. Since $\partial_xu(\cdot ,t_0)\in H^1(\mathbb R)$ one gets the desired result.

\end{proof}

The proofs of  Theorems \ref{A2} - \ref{A2m} part (i) are similar, so they will be omitted. 

\vskip.1in

\begin{proof} [Proof of Theorem \ref{A2} (ii)] In this case we work directly with the general form of the equation in \eqref{fNLS}. By hypotheses $u_2\equiv 0$  and from the equation \eqref{fNLS} it follows that in $\Omega\subset \R^n\times[0,T]$
\begin{equation}
\label{z1}
\begin{aligned}
u_1(x,t)&=(\mathcal L_m)^{\alpha/2}u_1(x,t)\\
&= (-\Delta+m^2)^{\alpha/2}u_1(x,t)=0,\;\;\,(x,t)\in \Omega.
\end{aligned}
\end{equation}
 Thus, there exists  $\,t_0\in(0,T)$ and an open (non-empty) set $\Theta\subset \mathbb R^n$ such that $\,\Theta\times\{t_0\}\subset \Omega$.

Hence, applying Theorem B if $m=0$ and Theorem \ref{C1} if $m>0$ to the function $\,u_1(x,t_0)\,$ one obtains that $\,u_1(x,t_0)=0$ for any $\,x\in\mathbb R^n$, which yields the desired result.

\end{proof}

The proof of Theorem  \ref{A2m} (ii) is similar to those described above, so it will be omitted. This is also the case of Corollary \ref{A18}.

\begin{proof} [Proof of Theorem \ref{X}] By hypothesis for any $(x,y)\in \Theta$
$$
u_1(x,y,t_0)=-(-\Delta)^{-1/2}\partial_y\theta(x,y,t_0)=0\;\;\;\text{and} \;\;\;\;\partial_y\theta(x,y,t_0)=0.
$$
Therefore, using the extension of Theorem B with $\alpha=-1/2$, see \cite {CoMoRa} and Remark \ref{rem1}, one has that
$$
\partial_y\theta(x,y,t_0)=0 \;\;\;\;\;\;(x,y)\in\mathbb R^2.
$$
Similarly, one gets that 
$$
\partial_x\theta(x,y,t_0)=0\;\;\;\;\;\;(x,y)\in\mathbb R^2.
$$
Since $\theta \in H^s(\mathbb R^2)$, $s>2,$ it follows that 
$$
\theta(x,y,t_0)=0\;\;\;\;\;\;(x,y)\in\mathbb R^2.
$$
Finally, using \eqref{qg1} we complete the proof.
\end{proof}
\vskip.1in 
\section{Appendix}

For $\alpha \in (-1,-1/2]$ we shall prove that if 
\begin{equation}
\label{app}
u\in C([0,T]:H^s(\mathbb R))\cap C^1([0,T]:H^{s-1}(\mathbb R))\;\;\;\text{with}\;\;s>3/2
\end{equation}
is a strong solution of the IVP associated to the equation \eqref{AA} and $x\partial_xu(x,0)\in L^2(\mathbb R)$ then
\begin{equation}
\label{desired}
x\partial_xu\in L^{\infty}([0,T]:L^2(\mathbb R)).
\end{equation}

First, we observe that if $f\in H^1(\mathbb R)$, then
\begin{equation}
\label{zz0}
x\partial_xD^{\alpha}(\partial_xf)= - (1+\alpha)D^{\alpha} (\partial_xf)+\partial_xD^{\alpha}(x\partial_xf).
\end{equation}

Therefore taking derivative of the equation \eqref{AA}, multiplying the result by $x^2\partial_xu$ and integration in the $x$-variable one gets
\begin{equation}
\label{energy}
\begin{aligned}
&\frac{1}{2} \frac{d\;}{dt}\int(x\partial_xu)^2(x,t)dx\\
&=\int x\partial_xD^{\alpha}\partial_x u\,x \partial_xu dx-\int x\partial_x(u\partial_xu)x\partial_xudx\\
&=E_1(t)+E_2(t).
\end{aligned}
\end{equation}

Integration by parts leads to the estimate
\begin{equation}
\label{zz1}
\begin{aligned}
|E_2(t)|&\leq c\| \partial_xu(t)\|_{\infty}\int(x\partial_xu)^2(x,t)dx\\
&+c\| u(t)\|_{\infty} \|\partial_xu(t)\|_2 (\int(x\partial_xu)^2(x,t)dx)^{1/2}.
\end{aligned}
\end{equation}
Thus, using \eqref{zz0} we write
\begin{equation}
\label{zz2}
E_1(t)= - (1+\alpha)\int D^{\alpha}\partial_xu \,x \partial_xu dx +\int \partial_xD^{\alpha}(x\partial_xu) \,x\partial_xu dx.
\end{equation}
Since $\partial_x D^{\alpha} $ is skew-symmetric the last term in \eqref{zz2} vanishes. Hence,
\begin{equation}\
\label{zz3}
|E_1(t)|\leq c\|\partial_xD^{\alpha}u(t) \|_2(\int(x\partial_xu)^2(x,t)dx)^{1/2}.
\end{equation}

Finally, inserting \eqref{zz1} and \eqref{zz3} in \eqref{energy} and using the hypothesis \eqref{app} one gets the desired result \eqref{desired}.

\vspace{5mm}
\noindent\underline{\bf Acknowledgements.} C.E.K.  was supported by the NSF grant DMS-1800082, D. P. was supported by a Trond Mohn Foundation grant. L.V. was supported by an
ERCEA Advanced Grant 2014 669689 - HADE, and by the MINECO
project PGC  2018-095422-B-100.



\begin{thebibliography}{99}


\bibitem{AbSe} M. J. Ablowitz and H. Segur, \emph{Solitons and the inverse scattering transform}. 
SIAM Studies in Applied Mathematics, 4. Society for Industrial and Applied Mathematics (SIAM), Philadelphia, Pa., 1981. 


\bibitem{BGLV} J. Bellazzini, V. Georgiev, E. Lenzmann, and N. Visciglia, \emph{On Traveling Solitary Waves and Absence of Small Data Scattering for Nonlinear Half-Wave Equations}, Comm. Math. Phys. {\bf 372} (2019), no. 2, 713--732

\bibitem{Be}  T. B. Benjamin, \emph{Internal waves of permanent
form in fluids of great depth},
J. Fluid Mech. {\bf 29} (1967), 559--592.
{\bf 31} (2013) 2--53.

\bibitem{BiHu} J. Biello and J. K. Hunter, \emph{Nonlinear Hamiltonian waves with constant frequency and surface waves on vorticity discontinuities}, Comm. Pure Appl. Math. {\bf 63} 2009, 303--336.



\bibitem{BoRi}  J. P. Borgna and D. F. Rial, \emph{Existence of ground states for a one-dimensional relativistic Schr\"odinger equation},
J. Math. Phys. {\bf 53} (2012), 062301, 19 pages.


\bibitem{CaSi} L. Caffarelli and L. Silvestre, \emph{An extension problem related to the fractional Laplacian},  Comm. P.D.E. {\bf 32} (2007), no. 7-9, 124--1260

\bibitem{CaHo} R. Camassa and D. Holm, {\em An integrable shallow water equation with peaked solitons}, Phys. Rev. Lett. {\bf 71} (1993), 1661--1664.

\bibitem{CHHO} Y. Cho, H. Hajaiej, G. Hwang, and T. Ozawa, \emph{On the Cauchy Problem for Fractional Schr\"odinger Equation with Hartree Type Nonlinearity}, Funkcialaj Ekvacioj {\bf 56} (2013), 193--224.



\bibitem{CoKeSt}J. Colliander, C. E. Kenig and G. Stafillani,
         \emph{Local well-posedness for dispersion generalized Benjamin-Ono
equations}, Diff. and Int. Eqs {\bf 16}, 12 (2002), 1441--1471.

\bibitem{CoMaTa} P. Constantin, A. J. Majda, and  E. Tabak, \emph{Formation of strong fronts in the 2-D quasi-geostrophic thermal active scalar},  Nonlinearity {\bf 7} (1994), no. 6, 1495--1533.


\bibitem{CoGSIo} D. Cordoba, J. Gomez-Serrano, and A. D. Ionescu, \emph{Global solutions for the generalized SQG patch equation},  Arch. Ration. Mech. Anal. 2{\bf 33}  (2019), no. 3, 1211--1251.


\bibitem{CoMoRa} G. Covi, K. M\"onkk\"onen, and J. Railo, \emph{Unique continuation property and Poincar\'e inequality for higher order laplacians with appliactions in inverse problems}, 2020 arXiv 2001.06210.






\bibitem{ElSc} A. Elgart and B. Schlein, {\em Mean field dynamics of boson stars}, Comm. Pure Appl. Math. {\bf 60} (2007), no. 4, 500--545

\bibitem{FaFe}  M. M. Fall and V. Felli, \emph {Unique continuation property and local asymptotics of solutions to fractional elliptic equations},  Comm. P.D.E. {\bf 39} (2014), 354--397.



\bibitem{FLP}  G. Fonseca, F. Linares, and G. Ponce, \emph{The IVP for the dispersion generalized Benjamin-Ono equation in weighted Sobolev spaces}, Ann. Inst. H. Poincar\'e Anal. Non Lin\'eaire {\bf 30} (2013), no. 5, 763--790. 

 \bibitem{FrLe} J. Fr\"ohlich and E. Lenzmann, \emph{Blow up for nonlinear wave equations describing boson stars}, Comm. Pure Appl. Math. {\bf 60} (2007), no. 11, 1691--1705.
  
  \bibitem{GGKM} C. S. Gardner, J. M. Greene, M. D. Kruskal, and R. M. Miura,
  \emph{Method for Solving the Korteweg-de Vries Equation}, 
  Phys. Rev. Lett.
 { \bf 19}
  (1967), 
  1095--1097.
  
  \bibitem{GeLePoRa}  P. G\'erard, E. Lenzmann, O. Pocovnicu, and P. Rapha\"el, 
\emph{A two-soliton with transient turbulent regime for the cubic half-wave equation on the real line} ,
Ann. PDE, {\bf 4} (2018), no. 1, Art. 7, 166 pp
  
  
  
   \bibitem{GiVe}  J. Ginibre, and G.  Velo, {\em Smoothing properties and existence of solutions for the generalized Benjamin-Ono equation},  J. Diff. Eqs. {\bf 93} (1991), 150--212.
   
   
   \bibitem{GhSaUh} T. Ghosh, M. Salo, and G. Uhlmann, \emph{The Calder\'on problem for the fractional Schr\"odinger equation}, preprint (2016), arxiv:1609.09248, to appear in Analysis \& PDE.
  
   \bibitem{Gu}  Z. Guo, {\em Local well-posedness for dispersion generalized Benjamin-Ono equations in Sobolev spaces},  J. Diff. Eqs {\bf 252} (2012), no. 3, 2053--2084.


  \bibitem{HIKK} S. Herr, A. Ionescu, C. E. Kenig, and H. Koch, {\em A para-differential renormalization technique for nonlinear dispersive equations},
Comm. P.D.E. {\bf 35} (2010), no. 10, 1827--1875. 
  
\bibitem{HoSi}   Y. Hong and Y. Sire, {\em  On fractional Schr\"odinger equations in Sobolev spaces},  Comm. Pure Appl. Anal. {\bf 14} (2015), no. 6, 226--2282.



\bibitem{HITW}  J. K. Hunter, M. Ifrim, D. Tataru, and T. K. Wong, \emph{Long time solutions for a Burgers-Hilbert equation via a modified energy method}, Proc. Amer. Math. Soc. {\bf 143} (2015), no. 8, 3407--3412 


 \bibitem{IoPu}  A. Ionescu and F. Pusateri, {\em Nonlinear fractional Schr\"odinger equations in one dimension},  J. Funct. Anal. {\bf 266} (2014), no. 1, 139--176.



\bibitem{KPV19} C. E. Kenig, G. Ponce, and L. Vega, \emph{Uniqueness Properties of Solutions to the Benjamin-Ono equation and related models}, preprint (2019), arXiv:1901.11432, to appear in J. Funct.  Anal.



\bibitem{KdV} D. J. Korteweg and G. de Vries
  \emph{On the change of form of long waves advancing in a
   rectangular canal, and on a new type of long stationary waves}, 
  Philos. Mag. 
 { \bf 39}
  (1895), 
  422--443.
  
\bibitem{KLR}  J.  Krieger, E. Lenzmann, and P. Rapha\"el, \emph{Nondispersive solutions to the $L^2$-critical half-wave equation},  Arch. Ration. Mech. Anal. {\bf 209} (2013) no. 1, 6--129.

\bibitem{La} N. Laskin, {\em Fractional Schr\"odinger equation}, Physical Rev. E. {\bf 66} (2002), no. 5, 056108, 7pp.


\bibitem{Le}  E. Lenzmann, {\em Well-posedness for semi-relativistic Hartree equations of critical type}, Math. Phys. Anal. Geom. {\bf 10} (2007), 43--62.


\bibitem{LiPo} F. Linares, and G. Ponce, {\em Unique Continuation Properties for solutions to the Camassa-Holm equation and other non-local equations},  preprint (2019) arXiv:1902.03279, to appear in Proc. A.M.S.




\bibitem{LiPiSa} F. Linares, D. Pilod, and J.-C. Saut,  {\em 
Dispersive perturbations of Burgers and hyperbolic equations I: Local theory}, 
SIAM J. Math. Anal. {\bf 46} (2014), no. 2, 1505--1537.





\bibitem{MoPiVe} L. Molinet, D. Pilod, and S. Vento,  {\em On well-posedness for some dispersive perturbations of Burgers' equation},  Ann. Inst. H. Poincar\'e, Anal. Non Lin\'eaire {\bf 35} (2018), no. 7, 1719--1756.

\bibitem{MoRi} L. Molinet and F. Ribaud, \emph {On global well-posedness for a class of nonlocal dispersive wave equations},  Discrete Cont. Dyn. Syst. {\bf 15} (2006), no. 2, 65--668.


\bibitem{MoSaTz} L. Molinet, J.C. Saut, and N. Tzvetkov, \emph{Ill-posedness issues for the Benjamin-Ono and related equations}, SIAM J. Math. Anal. {\bf 33} (2001), 982-988.

\bibitem{MoVe}  L. Molinet and S. Vento, \emph {Improvement of the energy method for strongly non-resonant dispersive equations and applications},  Anal. PDE {\bf 8} (2015), no. 6, 145--1495.

\bibitem{On} H. Ono,  \emph{Algebraic solitary waves in stratified fluids}, Journal  Physical Society of Japan, {\bf 39} (4) (1975), 108--1091.


\bibitem{Pe} J. Pedlosky, \emph{Geophysical Fluid Dynamics},  (New York Springer) (1987) 345--368, 653--670.


\bibitem{Ri} M. Riesz, \emph{Integrales de Riemann-Liouville et potentiels}, Acta Sci. Math. (Szeged) 9:1-1 (1938-40), 1--42.


\bibitem{Ru} A.  R\"uland, \emph{Unique continuation for fractional Schr\"odinger equations with rough potentials},  Comm. P.D.E. {\bf 40} (2015), no. 1, 77--114.



\bibitem{SaSc} J-C. Saut and B. Scheurer, {\em Unique continuation for some evolution equations}, J. Diff. Eqs. {\bf 66} (1987), 118--139.

\bibitem{ShVo} V. I. Shrira and V. V. Voronovich, {\em Nonlinear dynamics of vorticity waves in the coastal zone}, J. Fluid Mech. {\bf 326} (1996), 181--203.

\bibitem{StTo} P. R. Stinga and J. L. Torrea, \emph{Extension problem and Harnack's inequality for some fractional operators},  Comm. P.D.E.  {\bf 35} (2010), no. 11, 2092--2122.



\bibitem{St} E. M. Stein,
\emph{The characterization of functions arising as potentials},  Bull. A.M.S. {\bf 67} (1961), 102--104.

\bibitem{Ta} S. Tarama,  \emph{Analytic solutions of the Korteweg-de Vries equation}, J. Math. Kyoto Univ. {\bf 44} (2004),  1-32


\bibitem{Yu} H. Yu, \emph{Unique continuation for fractional orders of elliptic equations},  Ann. PDE {\bf 3} (2017), no. 2, Art. 16.

\end{thebibliography}
\end{document}